\documentclass[11pt]{article}
\usepackage{setspace,xargs}
\usepackage{amsmath}
 \usepackage[table]{xcolor}
\usepackage{graphicx,subfigure}
\usepackage{graphicx,xcolor}
\graphicspath{{./images/}}
 \usepackage{epsfig}
\usepackage{varwidth}
\usepackage [a4paper, bindingoffset=-.5cm, footskip=1cm, headheight = 16pt, top=3cm, bottom=2.5cm, right=3cm, left=3cm ] {geometry}
\input{amssym}
\input{epsf}
\usepackage{stackrel}
\usepackage{algorithm}
\usepackage[compatible]{algpseudocode}

\usepackage{float}
\usepackage{graphicx}
\usepackage{float}
\usepackage{morefloats}
\usepackage{float}
\usepackage{morefloats}
\usepackage{mathtools}
\usepackage{multirow}

 \usepackage[usenames,dvipsnames]{pstricks}
 \usepackage{epsfig}
 \usepackage{pst-grad} 
 \usepackage{pst-plot} 
 \usepackage[space]{grffile} 
 \usepackage{etoolbox} 
 \makeatletter 
 \patchcmd\Gread@eps{\@inputcheck#1 }{\@inputcheck"#1"\relax}{}{}
 \makeatother
\usepackage[pagewise]{lineno}
\newtheorem{prealg}{{\bf Algorithm}}

\newtheorem{prethm}{{\bf Theorem}}

\newenvironment{theorem}{\begin{prethm}{\hspace{-0.5
em}{\bf.}}}{\end{prethm}}
\newtheorem{preex}{{\bf Example}}

\newtheorem{prepr}{{\bf Problem}}

\newtheorem{precon}{{\bf Conjecture}}

\newenvironment{conjecture}{\begin{precon}{\hspace{-0.5
em}{\bf.}}}{\end{precon}}
\newtheorem{prelem}{{\bf Lemma}}

\newenvironment{lemma}{\begin{prelem}{\hspace{-0.5
em}{\bf.}}}{\end{prelem}}
\newtheorem{precor}{{\bf Corollary}}

\newtheorem{prepos}{{\bf Proposition}}

\newtheorem{preobv}{{\bf Observation}}

\newtheorem{predef}{{\bf Definition}}

\newtheorem{preproof}{{\bf Proof.}}

\newenvironment{proof}[1]{\begin{preproof}{\rm
               #1}\hfill{$\rule{2mm}{2mm}$}}{\end{preproof}}
               
               \newtheorem{preprooft}{{\bf Proof of Theorem 1.}}

\begin{document}
\renewcommand{\baselinestretch}{1.3}
\title{{\Large\bf  Star edge coloring of generalized Petersen graphs }}

{\small
\author{
{\sc Behnaz Omoomi$^a$} 
\ and
{\sc Marzieh Vahid Dastjerdi$^b$}
\\ [1mm]
{\small \it  $^a$ Department of Mathematical Sciences}\\
{\small \it  Isfahan University of Technology} \\
{\small \it 84156-83111, \ Isfahan, Iran}\\
{\small \it  $^b$ Research Center for Development of Advanced Technologies}\\
{\small \it  K. N. Toosi University of Technology} \\
{\small \it Tehran, Iran}}
\maketitle
\begin{abstract}
\noindent 
The star chromatic index of  a graph $G$, denoted by $\chi^\prime_s(G)$, is the smallest integer $k$ for which $G$ admits a proper edge coloring with $k$ colors such that  every path and cycle of length four  is not bicolored.  Let $d$ be the greatest common divisor of $n$ and $k$. 
Zhu~et~al. (\footnotesize{Discussiones Mathematicae: Graph Theory, 41(2): 1265, 2021}) showed that for every integers $k$ and $n> 2k$ with $d\geq 3$, generalized Petersen graph $GP(n,k)$ admits a 5-star edge coloring,  with the exception of the case that $d = 3$, $k\neq d$  and $\frac{n}{3}= 1\pmod{3}$.
Also, they conjectured that for every $n>2k$, $\chi^\prime_s(GP(n,k))\leq 5$, except $GP(3,1)$. 
In this paper, we prove that for every $GP(n,k)$ with $n\geq 2k$ and  $d\geq 3$ their conjecture is true. In fact, we  provide a 5-star edge coloring of $GP(n,k)$, where $n\geq 2k$ and  $d\geq 3$. We also obtain some results for 5-star edge coloring of $GP(n,k)$ with   $d=2$.
Moreover, Dvo{\v{r}}{\'a}k et al. ({\footnotesize Journal of Graph Theory, 72(3):313-326, 2013}) conjectured that the star of chromatic index of  subcubic graphs is at most 6. 
Thus, our results also prove this conjecture for the generalized Petersen graphs, as a class of subcubic graphs.
 
\end{abstract}

\section{Introduction}
A {\it proper edge coloring} of a graph $G$, is an assignment of colors to the edges of $G$ such that every two adjacent edges receive different colors.
 Under additional constraints on the proper  edge coloring, we get a variety of edge coloring, namely star edge coloring. 
In 2008, Liu and Deng~\cite{delta} introduced the concept of  star edge coloring.
 A {\it k-star edge coloring} of a graph is a proper edge coloring with at~most $k$ colors such that every path and cycle with four edges are not bicolored. 
 The smallest integer $k$ for which $G$ admits a $k$-star edge coloring is called the {\it star chromatic index} of $G$ and is denoted by $\chi^\prime_s(G)$~\cite{delta}.

Dvo{\v{r}}{\'a}k et al.~\cite{mohar} presented some upper bounds and lower
bounds on the star chromatic index of complete graphs and  derived a near linear upper bound in terms
of the maximum degree $\Delta$ for general graphs. Also, they showed that the
star chromatic index of cubic graphs lies between 4 and 7.  Moreover, they 
proposed the following conjecture.
\begin{conjecture} \rm{\cite{mohar}} \label{conj}
If $G$ is a subcubic graph \rm{(}graph with maximum degree three\rm{)}, then $\chi^\prime_s(G)\leq 6$. 
\end{conjecture}
Many attempts to prove this conjecture have been made  and the quest  go on.
 In~\cite{class}, Bezegov{\'a}~et~al. obtained some bounds on the star
chromatic index of subcubic outerplanar graphs, trees and outerplanar graphs. They showed that $\chi^\prime_s(G) \leq 5$, where $G$ is a
subcubic outerplanar graph. In~\cite{Omoomi1},    the star chromatic index of two-dimensional grid graphs, also confirmed   this conjecture.

In 2018, Lei et al. \cite{Lei0,Lei} consider the maximum average degree of subcubic graphs and obtained some results. The {\it maximum average degree} of a graph $G$, denoted $mad(G)$, is defined as the maximum of $2|E(H)|/|V(H)|$ taken over all the subgraphs $H$ of $G$.
In \cite{Lei0,Lei}, Lei et al. proved that if $G$ is a subcubic graph, then
 \[\chi^\prime_s(G)\leq\begin{cases}
 4~~~mad(G)<2,\\
 5~~~mad(G) < 12/5,\\
 6~~~mad(G)<5/2.
 \end{cases}
 \]

In this paper, we are going to study the star chromatic index of generalized Petersen graphs as a class of subcubic graph.
 Various coloring of generalized Petersen graphs,   such as acyclic coloring, strong and injective edge coloring, have been studied in \cite{Chen,Li,Zhu2}.
For integers $n$ and $k$ with $n\geq 2 k$, the {\it generalized Petersen graph} $GP(n, k)$ has  vertices $\{u_0,\ldots,u_{n-1}\}\cup \{v_0,\ldots,v_{n-1}\}$ and edges $(\bigcup_{i=0}^{n-1}{u_iu_{i+1}})\cup(\bigcup_{i=0}^{n-1}{v_iv_{i+k}})\cup(\bigcup_{i=0}^{n-1}{u_iv_{i}})$, where indices of vertices are modula $n$. 
 Each edge of  $\{u_iv_{i}: 0\leq i\leq n-1\}$ is called an {\it spoke} of $GP(n,k)$. Let $d$ be the greatest common divisor $n$ and $k$ (or $d=GCD(n,k)$). The value of $d$ determines the number of disjoint cycles (or disjoint edges when $n=2k$) on $\{v_0,\ldots, v_{n-1}\}$. Moreover, if $n>2k$, then  $GP(n,k)$ is a cubic graph. For $n=2k$,  the vertices of $\{v_0,\ldots,v_{n-1}\}$ are of degree two and $GP(n,k)$ is a subcubic graph.

 In 2021, Zhu et al. studied the star chromatic index of cubic generalized Petersen graphs in~\rm{\cite{Zhu}}. 
 They proved that $\chi^\prime_s(GP(n, k)) = 4$ if and only if $n=0 \pmod{4}$ and $k$ is odd.
 They also showed that  for every two integers $n$ and $k$, $n>2k$, such that
$d\geq 3$, $GP(n, k)$ has a 5-star edge coloring, with the exception of the
case that $d= 3$, $k\neq3$ and $\frac{n}{3}=1  \pmod 3$.
Moreover, they obtained some partial results for 5-star edge coloring of $GP(n,k)$, in the following cases.
\begin{itemize}
\item
 $n = 0 \pmod 2$, $k = 1 \pmod 2$, and $d= 1$. 
\item
$n\geq 5$ and $k=1$.
\item
$n=0\pmod 6$, and $k=2$.
\end{itemize}
 Furthermore, Zhu et al.  found  that $\chi^\prime_s(GP(3,1))=6$ and conjectured that $GP(3,1)$ is the unique generalized Petersen graph that admits no 5-star edge coloring.
\begin{conjecture} {\rm \cite{Zhu}}\label{conj2}
If $GP(n,k)\neq GP(3,1)$, then $\chi^\prime_s(GP(n,k))\leq 5$. 
\end{conjecture}
Let $t$  be the minimum positive integer  that $tk=d\pmod{n}$. 
In this paper, we consider $GP(n,k)$ with $n\geq 2k$, and prove  Conjectures~\ref{conj2} and Conjecture~\ref{conj},  for the generalized Petersen graphs with the following parameters.
\begin{itemize}
\item
 $d\geq 3$ (contains Zhou et al.'s exception for $d=3$).
\item
$d=2$, and 
\begin{itemize}
\item 
$n = 0 \pmod 6$.
\item
$n=2\pmod 6$, $t=2\pmod 3$.
\item
$n=4\pmod 6$, $t=1\pmod 3$.
\end{itemize}
\item
 $\frac{n}{d}\in \{2,5\}$.
\end{itemize} 
This paper is organized as follows. In Section~\ref{pre}, among some terminology and notations, we prove some useful lemmas for extending the  partial star edge coloring of paths and cycles. In Section~\ref{main}, we consider $GP(n,k)$ with $d\geq 2$. We first  present patterns for 5-star edge coloring of $G(n,k)$ which $n=5$ or $\frac{n}{d}\in\{2,5\}$. Then, we prove $\chi^\prime_s(GP(n,k)\leq 5$ when $d\geq 3 $. Finally,  we present some partial results for 5-star edge coloring of $GP(n,k)$, where $d=2$.

\section{Partial star edge coloring of paths and cycles}\label{pre}
This section contains the terminology and notations that we need through the paper.  Moreover, we prove some useful lemmas which are used for coloring  some paths and cycles  of  a given generalized Petersen graph in Section~3.

The {\it distance of two vertices} of $G$ is the number of edges in a shortest path connecting them.
 The {\it line graph} of a graph $G$, $L(G)$, is a graph whose vertices correspond to the edges of $G$, and every two vertices of  $L(G)$ are adjacent if and only if the corresponding edges of $G$ are incident with the same vertex of $G$.  The {\it distance of two edges} of $G$  is equal to their corresponding vertex distance in  the line graph of $G$.

 We demonstrate a path or cycle by the sequence of its vertices or its edges. For example,   $P:=x_0,x_1,\ldots,x_{n-1}$ is a path, and $C:=x_0,x_1,\ldots,x_{n-1},x_0$ is a cycle on vertex set $\{x_0,\ldots,x_{n-1}\}$. If we use notations   $e_i:=x_ix_{i+1}$, $0\leq i\leq n-1$, for the edges of $P$ and $C$, then we could denote $P:=e_0,e_1,\ldots,e_{n-2}$ and $C:=e_0,e_1,\ldots,e_{n-1}$.  Note that  the subscripts of vertices in every cycle  or path, are taken modulo size of its vertex set) . The {\it length} of a path or a cycle is the number of its edges. For every integer $k\geq 1$,  {\it $k$-path} ($k$-cycle) is a path (cycle) of length $k$. 
 
  Consider a direction (clockwise or anti-clockwise) on $C$ or $P$ with $n$ edges. For every edge $e$ of  $C$ (or $P$) and  positive integer $r<n$, $e+r$ denotes the   $r$-th edge after $e$ on $C$ (or $P$) and $e-r$ denotes  the   $r$-th edge before $e$ on $C$ (or $P$). Note that for path $P$, the value of $r$ should be such that $e-r,e+r\in P$.

For every edge coloring $f$ of  a graph $G$,  we define some notations as follows.
\begin{enumerate}
\item[$\bullet$]
For every vertex $x$ of $G$, $\mathcal{F}_G(x)$ is the set of colors of the edges incident to~$x$.
\item[$\bullet$]
For every edge $xy$ of $G$, $\mathcal{F}_G(xy)$ is the set of colors of the edges adjacent to $xy$.

\item[$\bullet$]
 For every path $P=x_0,x_1,\ldots,x_n$ or cycle $C=x_0,x_1,\ldots,x_n,x_0$, we denote the sequence of colors of  edges by $f(P)$ and $f(C)$ which
\begin{align*} 
&f(P):=f(x_0x_1),f(x_1x_2),\ldots,f(x_{n-1}x_n),\\
&f(C):=f(x_0x_1),f(x_1x_2),\ldots,f(x_{n-1}x_n),f(x_{n}x_{0}).
\end{align*}
\end{enumerate}

\begin{lemma}\label{lem1}
Let $C=x_0,\ldots,x_{n-1},x_0$ be a cycle of length $n\neq 5$, and $f$ is a partial edge coloring of at most 6 edges of $C$ such that $f(x_0,x_1,x_2,x_3)=a,b,c$ and for some $0\leq i< n$,

\[f(x_i,x_{i+1},x_{i+2},x_{i+3})=
\begin{cases}
a,b,c~~~&i=0,\\
b,c,b~~~& i=1,\\
c,b,a~~~&i\in\{2,n-2\}\\
b,a,b~~~&i=n-1,\\
a,c,b~~~&otherwise.
\end{cases}
\]
Then, the partial edge coloring $f$ of $C$ can be extended to a 3-star edge coloring of $C$.
\end{lemma}
\begin{proof}{
 Let $P:=x_0,x_1,x_2,x_3$ and $P^\prime:=x_i,x_{i+1},x_{i+2},x_{i+3}$, $0\leq i\leq n-1$, are two colored subpath of $C$.
 We consider the following cases, and in each case extend the given partial star edge coloring  $f$ to a star edge coloring of $C$ with colors of  $\{a,b,c\}$.

\noindent\textbf{Case 1:} $i=0$.\\
\[f(C)=
\begin{cases}
\underbrace{\underline{ a,b,c},\ldots,\underline{a,b,c}}_\text{n} &n=0 \pmod 3,\\
\underbrace{\underline{ a,b,c},\ldots,\underline{a,b,c}}_\text{n-1},b &n=1  \pmod 3,\\
\underbrace{\underline{ a,b,c},\ldots,\underline{a,b,c}}_\text{n-5},a,b,a,c,b &n=2  \pmod 3.
\end{cases}
\]
\textbf{Case 2:} $i\in\{1,n-1\}$.\\
The partial edge colorings are $f(x_0,x_1,\ldots, x_4)=a,b,c,b$  and $f(x_{n-1},x_0,\ldots,x_3)=b,a,b,c$, for $i=1$ and $i=n-1$, respectively. Since the giving  partial star edge coloring in both cases contains only a bicolored 3-path,  the extension of $f$ in both cases can be concluded with the same argument. Thus, without loss of generality, we assume that $i=1$ and extend the partial star edge coloring $f$ such that
 \[f(C)=
\begin{cases}
a,b,\underbrace{\underline{ c,b,a},\ldots,\underline{c,b,a}}_\text{n-3},c &n=0  \pmod 3,\\
a,b,c,b,\underbrace{\underline{ a,c,b},\ldots,\underline{a,c,b}}_\text{n-4} &n=1  \pmod 3,\\
a,b,c,b,\underbrace{\underline{ a,c,b},\ldots,\underline{a,c,b}}_\text{n-5},c &n=2  \pmod 3,\\
\end{cases}
\]

\noindent\textbf{Case 3:} $i\in\{2,n-2\}$.\\
In this case, the partial edge colorings are  $f(x_0,x_1,\ldots, x_5)=a,b,c,b,a$ and $f(x_{n-2},x_{n-1},\ldots,x_3)=c,b,a,b,c$, for $i=2$ and $i=n-2$, respectively. Since $C$ has a partial 3-star edge coloring of a 5-path  in both cases with the same pattern by renaming the colors,  we only consider the case $i=2$; the case $i=n-2$  follows similarly. Thus, for $i=2$ we extend  $f$ such that


\[f(C)=
\begin{cases}

a,b,\underbrace{\underline{ c,b,a},\ldots,\underline{c,b,a}}_\text{n-3},c & n=0  \pmod 3,\\
a,b,\underbrace{\underline{ c,b,a},\ldots,\underline{c,b,a}}_\text{n-4},c,b & n=1  \pmod 3,\\
a,b,\underbrace{\underline{ c,b,a},\ldots,\underline{c,b,a}}_\text{n-5},c,b,c & n=2  \pmod 3.
\end{cases}
\]
\textbf{Case 4:} $2<i<n-2$.\\
 In this case, $P$ and $P^\prime$ are edge disjoint  paths with coloring patterns
  $f( x_0,x_1,\ldots, x_3)=a,b,c$ and $f(x_i,x_{i+1},\ldots,x_{i+3})=a,c,b$. We provide a 3-star edge coloring of $C$ as follows.
\[f(x_0,x_1,\ldots,x_{i})=
\begin{cases}
\underbrace{\underline{ a,b,c},\ldots,\underline{a,b,c}}_\text{i} &i=0 \pmod 3,\\
\underbrace{\underline{ a,b,c},\ldots,\underline{a,b,c}}_\text{i-1},b &i=1 \pmod 3,\\
\underbrace{\underline{ a,b,c},\ldots,\underline{a,b,c}}_\text{i-2},a,b & i=2 \pmod 3.
\end{cases}
\]

\[~~f(x_{i},x_{i+1},\ldots,x_{0})=
\begin{cases}
\underbrace{\underline{ a,c,b},\ldots,\underline{a,c,b}}_\text{n-i} &n-i=0 \pmod 3,\\
\underbrace{\underline{ a,c,b},\ldots,\underline{a,c,b}}_\text{n-i-1},c &n-i=1 \pmod 3,\\
\underbrace{\underline{ a,c,b},\ldots,\underline{a,c,b}}_\text{n-i-2},a,c& n-i=2 \pmod 3.
\end{cases}
\]
It can be easily seen that the given edge coloring is a 3-star edge coloring of $C$.}
\end{proof}


\begin{lemma}\label{lem:7-path}
Let  $P:=x_0,x_1,\ldots,x_n$, $n>4$,  be a path    with a partial 3-star edge  coloring $f$ of subpaths $x_0,x_1,x_2$  and $x_4,\ldots,x_n$. The giving partial edge coloring can be extend to a 3-star edge coloring of $P$ if  and only if at least one of the following condition holds.
\begin{itemize}
\item[\rm(i)]
$n\leq 5$.
\item[\rm(ii)]
$f(x_4x_5)\in \mathcal{F}_P(x_1)$.
\item[\rm{(iii)}]
$f(x_5x_6)\neq f(x_0x_1)$.
\end{itemize}

\end{lemma}

\begin{proof}{
Assume that the color set is $\{a,b,c\}$ and  $f(x_0,x_1,x_2)=a,b$.
If $n=5$, then  we have three colored edges $\{x_0x_1,x_1x_2,x_4x_5\}$ and we can set 
\[f(x_2x_3)=c, ~~~f(x_3x_4)=\begin{cases}
b~~~&f(x_4x_5)=a,\\
a~~~&otherwise.
\end{cases}
\]
For $n\geq 6$ it is suffices to show that the partial edge coloring of $P$ can not be extended to a 3-star edge coloring of $P$ if and only if  $f(x_4,x_5,x_6)=c,a$.
First, assume that  $f(x_4,x_5,x_6)\neq c,a$. We choose $f(x_3x_4)$ from $\{a,b,c\}\setminus\mathcal{F}_P(x_5)$. There are  two cases; $b\in\mathcal{F}_P(x_5)$ or $f(x_4,x_5,x_6)\neq a,c$.  In the first case,  $f(x_3x_4)\neq b$  and we can choose $f(x_2x_3)$ from $\{a,c\}\setminus f(x_3x_4)$.  In the second case, $f(x_3x_4)=b$. Hence, we set $f(x_2x_3)=c$. In both cases, we have a 3- star edge coloring of $P$.

Now suppose that  $f(x_4,x_5,x_6)=c,a$. If we set $f(x_3x_4)=a $ or  $f(x_3x_4)=b$, then there is no choice for $f(x_2x_3)$. Thus, the partial edge coloring can not be extended to a 3-star edge coloring of $P$.
}\end{proof}

By considering  Lemma \ref{lem:7-path} and its proofs, we propose two Process  to extend the given partial edge coloring $f$ of path $P:=x_0,x_1,\ldots,x_n$ to a  3-star edge coloring.  

\noindent$\bullet$ {\bf Process 1:}  If at least two last edges  at one side of $P$ (e.g $x_0x_1$ and $x_1x_2$ in the left side) and exactly one edge at the other side ($x_{n-1}x_n$ in the right side) are colored by $f$, then we extend $f$ to $f_1$  by starting coloring  from the uncolored edge  on the side with more colored edges (left side)  and moving to the another side (right side). We choose the color of each uncolored edge  $e$ of the path such that
\[f_1( e)\in \begin{cases}
  \{a,b,c\}\setminus \{f_1(e-1),f_1(e-2)\},~~~&|\mathcal{F}_P(e)|=1,\\
  \{a,b,c\}\setminus \mathcal{F}_P(e)~~~&otherwise.\\
  \end{cases}
  \]


\noindent$\bullet$ {\bf Process 2:} Let  two edges in each end of $P$ are colored  by $f$ and the direction for moving is given (clockwise or anti-clockwise). Then, we start coloring from the first uncolored edge based on the given direction and extend $f$ to $f_2$  such that

 \[f_2(e)\in
\begin{cases}
\{a,b,c\}\setminus\{f_2(e-1),f_2(e-2)\}~~~&|\mathcal{F}_P(e)|=1, \text{edge}~ e+3~\text{is uncolored},\\
\{f_2(e+3),f_2(e+4)\}\setminus\{f_2(e-1)\}~~~& \text{edges}~e+3,e+4~\text{are already colored},\\
\{a,b,c\}\setminus \mathcal{F}_P(e)~~~&otherwise.
\end{cases}
\]

\begin{lemma}\label{lem2}
 Let $f$ be a partial proper edge coloring  of three edges of cycle   $C:=x_0,x_1,\ldots x_{n-1},x_0$ with $n\neq 5$.  If   at least two colored edges are  adjacent, then $f$ could be extend to a 3-star edge coloring of $C$. 
\end{lemma}

\begin{proof}
{Without loss of generality, suppose that  path $x_0,x_1,x_2$ and edge $x_ix_{i+1}$, for some $2\leq i\leq n-2$, are colored by $f$ and $f(x_0,x_1,x_2)=a,b$. Consider paths $P_1:=x_2,x_3,\ldots,x_{i}$ and $P_2:=x_{i+1},x_{i+2}\ldots,x_{0}$. 
  To provide a 3-star edge coloring of $C$, we first color the shorter path. Since $n\neq 5$, the length of longer path is at least two.
  
 First,  assume that  $P_1$ is the shorter path.
  We extend the partial edge coloring of $x_0,x_1,P_1,x_{i+1}$ clockwise by Process 1.  
 With a simple check, we can see that at most one  bicolored subpath appears in $C$ after applying Process 1. 
     If $x_{0},x_{1},x_{2},x_3$ is bicolored,  we extend the partial edge coloring of $x_{i-1},x_i,P_2,x_1,x_2$ anti-clockwise by Process 2; Otherwise, clockwise by Process 2. 
     Note that by Lemma~\ref{lem:7-path},  if the length of $P_2$ is greater than two or $f(x_ix_{i+1})\neq  c$, the partial edge coloring of $x_{i-1},x_i,P_2,x_1,x_2$ could be extended. 
     If the length of $P_2$ is two and    $f(x_ix_{i+1})=  c$, then  by applying the above coloring we obtain the following coloring of $C$ which obviously is a 3-star edge coloring.
\[f(C)=\begin{cases}
a,b,a,c,b,c~~~&n=6,\\
a,b,c,a,c,b,c~~~&n=7.
\end{cases}
\]

If $P_2$ is shorter, with the similar argument, we could obtain a 3-star edge coloring of $C$.
}
\end{proof}

\begin{lemma}\label{lem3}
 Let  $C:=x_0,x_1,\ldots x_{n-1},x_0$ be a cycle of length  $n\neq 5$. Then, there exists a 3-star edge coloring $f$ of $C$ with colors $\{a,b,c\}$ such that  for five  edges $\{x_0x_1,x_{i-1}x_i,x_ix_{i+1},x_{j-1}x_{j},x_{j}x_{j+1}\}$, we have
 \begin{itemize}
 \item
 $f(x_0x_1)\in\{a,b,c\}$.
 \item
  $2\leq i<j\leq n-2$, $|i-j|\geq 2$ is even, and $\mathcal{F}_C(x_i)=\{a,b\}$, $\mathcal{F}_C(x_j)=\{b,c\}$. 
 \end{itemize}
\end{lemma}

\begin{proof}
{ Consider three subpaths $P_1:=x_{i+1},x_{i},\ldots,x_1,x_0$, $P_2:=x_{i-1},x_i,\ldots,x_{j+1}$,  and $P_3:=x_{j-1},x_{j},\ldots,x_{n-1},x_0,x_1$ of $C$.
If $|i-j|=2$ , then we set $f(P_2)=b,a,c,b$. Otherwise, we set $f(x_{i-1},x_i,x_{i+1},x_{i+2})=a,b$ and $f(x_{j-2},x_{j-1},x_j,x_{j+1})=b,c$.
Now, we choose the shorter path of $P_1$ and $P_3$.  First, extend edge coloring of the shorter path clockwise by Process 1  and then extend the coloring of the longer one  anti-clockwise by Process 2. 

If $|\mathcal{F}_C(x_j)|=1$, then we extend the 3-star edge coloring of $P_2$  clockwise by  Process 2. Otherwise, we extend its partial edge coloring anti-clockwise by Process 2.

Note that, by Lemma~\ref{lem:7-path}, all of the above extensions are possible using colors $\{a,b,c\}$. Thus, we obtain a 3-star edge coloring of $C$.
}
\end{proof}

\section{Star edge coloring of $GP(n,k)$}\label{main}
In this section, we consider  generalized Petersen graph  $GP(n,k)$, $n\geq 2k$, as a  union of  some cycles   and spokes.  Then, in several steps, we  provide a 5-star edge coloring of $GP(n,k)$ with $GCD(n,k)\geq 3$. Moreover,  we present some partial results for 5-star edge coloring of $GP(n,k)$, when $GCD(n,k)=2$.


Let  $GP(n,k)$ be a generalized Petersen graph (see Figure~1) with cycles $C^0,C^1,\ldots,C^d$, where
\begin{itemize}
\item $d=GCD(n,k)$, and $t$ is the minimum positive integer such that $tk=d\pmod{n}$.
\item $C^0:=u_0,\ldots,u_{n-1}$, $C^i=v^i_{0},\ldots,v^i_{\frac{n}{d}-1}$, for $1\leq i\leq d$.
\item for every $1\leq i\leq d$ and $0\leq r<\frac{n}{d}$,   vertices $v^i_{r}$  and $u_{i-1+rk}$ are matched by an spoke  $s^i_{r}$.  
\item vertex $u_{\ell}$, $\ell=k(\frac{n}{d}-1)\pmod{n}$, is the end of spoke $s^1_{\frac{n}{d}-1}$ in $C^0$.
\item we call the edges $u_{rd-1}u_{rd}$, $0\leq r\leq \frac{n}{d}-1$,  as \it{connector edges}. Each connector edge  of $C^0$, connects a spoke with an end on $C^1$ to  another spoke with an end on $C^d$. More precisely, for every $0\leq r<\frac{n}{d}$, $s^1_r$ and $s^d_{r-t}$ are  adjacent to a connector. 
\end{itemize} 
Note that, for  every  $0\leq r<\frac{n}{d}$ and $1\leq i\leq d$,  the ends of  $s^i_r$ and $s^i_{r+1}$  in $C^0$, are  in distance  $k$ . Moreover, the ends of  $s^i_r$ and $s^i_{r+t}$  in $C^0$, are in distance $d$.
%
\begin{figure}[H]\label{fig1}

\centering
\psscalebox{0.75 0.75} 
{
\begin{pspicture}(1,-3)(20.704,3)

\psline[linecolor=black, linewidth=0.04](1.8101643,0.80857015)(2.6358073,0.80857015)
\psline[linecolor=black, linewidth=0.04, linestyle=dotted, dotsep=0.10583334cm](2.8717053,0.80857015)(3.8152974,0.80857015)
\psline[linecolor=black, linewidth=0.04](3.9332464,0.80857015)(4.8768387,0.80857015)
\psellipse[linecolor=black, linewidth=0.04, fillstyle=solid,fillcolor=black, dimen=outer](2.23,-1.4360001)(0.15,0.15)
\psellipse[linecolor=black, linewidth=0.04, fillstyle=solid,fillcolor=black, dimen=outer](1.7276,0.7740905)(0.17692351,0.1723982)
\psellipse[linecolor=black, linewidth=0.04, fillstyle=solid,fillcolor=black, dimen=outer](3.9686313,0.7740905)(0.17692351,0.1723982)
\psellipse[linecolor=black, linewidth=0.04, dimen=outer](2.85026,0.777225)(1.5011692,0.47017688)
\psline[linecolor=black, linewidth=0.04](4.9947877,0.80857015)(5.8204308,0.80857015)
\psline[linecolor=black, linewidth=0.04, linestyle=dotted, dotsep=0.10583334cm](6.056329,0.80857015)(6.999921,0.80857015)
\psline[linecolor=black, linewidth=0.04](7.11787,0.80857015)(8.061461,0.80857015)
\psellipse[linecolor=black, linewidth=0.04, fillstyle=solid,fillcolor=black, dimen=outer](5.8558154,0.7740905)(0.17692351,0.1723982)
\psellipse[linecolor=black, linewidth=0.04, fillstyle=solid,fillcolor=black, dimen=outer](4.9122233,0.7740905)(0.17692351,0.1723982)
\psellipse[linecolor=black, linewidth=0.04, fillstyle=solid,fillcolor=black, dimen=outer](7.1532545,0.7740905)(0.17692351,0.1723982)
\psellipse[linecolor=black, linewidth=0.04, dimen=outer](6.0348835,0.777225)(1.5011692,0.47017688)
\psline[linecolor=black, linewidth=0.04](11.481983,0.80857015)(12.307626,0.80857015)
\psline[linecolor=black, linewidth=0.04, linestyle=dotted, dotsep=0.10583334cm](12.543524,0.80857015)(13.487116,0.80857015)
\psline[linecolor=black, linewidth=0.04](13.605065,0.80857015)(14.548657,0.80857015)
\psellipse[linecolor=black, linewidth=0.04, fillstyle=solid,fillcolor=black, dimen=outer](12.34301,0.7740905)(0.17692351,0.1723982)
\psellipse[linecolor=black, linewidth=0.04, fillstyle=solid,fillcolor=black, dimen=outer](11.399419,0.7740905)(0.17692351,0.1723982)
\psellipse[linecolor=black, linewidth=0.04, fillstyle=solid,fillcolor=black, dimen=outer](13.6404505,0.7740905)(0.17692351,0.1723982)
\psellipse[linecolor=black, linewidth=0.04, dimen=outer](12.5220785,0.777225)(1.5011692,0.47017688)
\psline[linecolor=black, linewidth=0.04, linestyle=dotted, dotsep=0.10583334cm](7.943513,0.80857015)(10.65634,0.80857015)
\psline[linecolor=black, linewidth=0.04](10.538391,0.80857015)(11.364034,0.80857015)
\psellipse[linecolor=black, linewidth=0.04, fillstyle=solid,fillcolor=black, dimen=outer](10.455827,0.7740905)(0.17692351,0.1723982)
\psline[linecolor=black, linewidth=0.04](17.73328,0.80857015)(18.558924,0.80857015)
\psline[linecolor=black, linewidth=0.04, linestyle=dotted, dotsep=0.10583334cm](18.79482,0.80857015)(19.738413,0.80857015)
\psellipse[linecolor=black, linewidth=0.04, fillstyle=solid,fillcolor=black, dimen=outer](18.594307,0.7740905)(0.17692351,0.1723982)
\psellipse[linecolor=black, linewidth=0.04, fillstyle=solid,fillcolor=black, dimen=outer](17.650717,0.7740905)(0.17692351,0.1723982)
\psellipse[linecolor=black, linewidth=0.04, fillstyle=solid,fillcolor=black, dimen=outer](19.891747,0.7740905)(0.17692351,0.1723982)
\psellipse[linecolor=black, linewidth=0.04, dimen=outer](18.773376,0.777225)(1.5011692,0.47017688)
\psline[linecolor=black, linewidth=0.04](16.789688,0.80857015)(17.615332,0.80857015)
\psline[linecolor=black, linewidth=0.04, linestyle=dotted, dotsep=0.10583334cm](14.548657,0.80857015)(17.261484,0.80857015)
\psellipse[linecolor=black, linewidth=0.04, fillstyle=solid,fillcolor=black, dimen=outer](14.584042,0.7740905)(0.17692351,0.1723982)
\psbezier[linecolor=black, linewidth=0.04](1.8101643,0.80857015)(0.71417844,1.2333242)(8.4424925,1.354)(10.32,1.3539999999999999)(12.197507,1.354)(20.892294,1.2696742)(19.80275,0.8294669)
\rput[bl](1.6214926,1.4727001){$u_0$}
\rput[bl](2.5650847,1.4727001){$u_1$}
\rput[bl](3.6886768,1.4727001){$u_{d-1}$}
\rput[bl](4.808167,1.4727001){$u_{d}$}
\rput[bl](10.115872,1.4727001){$u_{\ell-1}$}
\rput[bl](11.295362,1.4727001){$u_{\ell}$}
\rput[bl](12.003057,1.4267273){$u_{\ell+1}$}
\rput[bl](13.135366,1.4267273){$u_{\ell+d-1}$}
\rput[bl](14.550755,1.4267273){$u_{\ell+d}$}
\rput[bl](19.536922,1.3756865){$u_{n-1}$}
\psellipse[linecolor=black, linewidth=0.04, fillstyle=solid,fillcolor=black, dimen=outer](1.55,-1.4360001)(0.15,0.15)
\psline[linecolor=black, linewidth=0.04](1.62,-1.406)(2.34,-1.406)
\psline[linecolor=black, linewidth=0.04, linestyle=dotted, dotsep=0.10583334cm](2.34,-1.406)(2.96,-1.406)
\psline[linecolor=black, linewidth=0.04](3.18,-1.406)(5.06,-1.406)
\psellipse[linecolor=black, linewidth=0.04, fillstyle=solid,fillcolor=black, dimen=outer](3.03,-1.4360001)(0.15,0.15)
\psellipse[linecolor=black, linewidth=0.04, fillstyle=solid,fillcolor=black, dimen=outer](3.75,-1.4360001)(0.15,0.15)
\psellipse[linecolor=black, linewidth=0.04, fillstyle=solid,fillcolor=black, dimen=outer](4.41,-1.4360001)(0.15,0.15)
\psellipse[linecolor=black, linewidth=0.04, fillstyle=solid,fillcolor=black, dimen=outer](5.07,-1.4360001)(0.15,0.15)
\psline[linecolor=black, linewidth=0.04, linestyle=dotted, dotsep=0.10583334cm](5.28,-1.406)(5.9,-1.406)
\psellipse[linecolor=black, linewidth=0.04, fillstyle=solid,fillcolor=black, dimen=outer](5.85,-1.4360001)(0.15,0.15)
\psline[linecolor=black, linewidth=0.04](6.0,-1.406)(6.72,-1.406)
\psellipse[linecolor=black, linewidth=0.04, fillstyle=solid,fillcolor=black, dimen=outer](6.61,-1.4360001)(0.15,0.15)
\psbezier[linecolor=black, linewidth=0.04](1.62,-1.406)(2.04,-1.106)(2.4,-0.866)(4.02,-0.8660000000000002)(5.64,-0.866)(6.18,-1.046)(6.54,-1.466)
\psellipse[linecolor=black, linewidth=0.04, fillstyle=solid,fillcolor=black, dimen=outer](8.11,-1.4360001)(0.15,0.15)
\psellipse[linecolor=black, linewidth=0.04, fillstyle=solid,fillcolor=black, dimen=outer](7.43,-1.4360001)(0.15,0.15)
\psline[linecolor=black, linewidth=0.04](7.5,-1.406)(8.22,-1.406)
\psline[linecolor=black, linewidth=0.04, linestyle=dotted, dotsep=0.10583334cm](8.22,-1.406)(8.84,-1.406)
\psline[linecolor=black, linewidth=0.04](9.06,-1.406)(10.94,-1.406)
\psellipse[linecolor=black, linewidth=0.04, fillstyle=solid,fillcolor=black, dimen=outer](8.91,-1.4360001)(0.15,0.15)
\psellipse[linecolor=black, linewidth=0.04, fillstyle=solid,fillcolor=black, dimen=outer](9.63,-1.4360001)(0.15,0.15)
\psellipse[linecolor=black, linewidth=0.04, fillstyle=solid,fillcolor=black, dimen=outer](10.29,-1.4360001)(0.15,0.15)
\psellipse[linecolor=black, linewidth=0.04, fillstyle=solid,fillcolor=black, dimen=outer](10.95,-1.4360001)(0.15,0.15)
\psline[linecolor=black, linewidth=0.04, linestyle=dotted, dotsep=0.10583334cm](11.16,-1.406)(11.78,-1.406)
\psellipse[linecolor=black, linewidth=0.04, fillstyle=solid,fillcolor=black, dimen=outer](11.73,-1.4360001)(0.15,0.15)
\psline[linecolor=black, linewidth=0.04](11.88,-1.406)(12.6,-1.406)
\psellipse[linecolor=black, linewidth=0.04, fillstyle=solid,fillcolor=black, dimen=outer](12.49,-1.4360001)(0.15,0.15)
\psbezier[linecolor=black, linewidth=0.04](7.5,-1.406)(7.92,-1.106)(8.28,-0.866)(9.9,-0.8660000000000002)(11.52,-0.866)(12.06,-1.046)(12.42,-1.466)
\psellipse[linecolor=black, linewidth=0.04, fillstyle=solid,fillcolor=black, dimen=outer](15.67,-1.4360001)(0.15,0.15)
\psellipse[linecolor=black, linewidth=0.04, fillstyle=solid,fillcolor=black, dimen=outer](14.99,-1.4360001)(0.15,0.15)
\psline[linecolor=black, linewidth=0.04](15.06,-1.406)(15.78,-1.406)
\psline[linecolor=black, linewidth=0.04, linestyle=dotted, dotsep=0.10583334cm](15.78,-1.406)(16.4,-1.406)
\psline[linecolor=black, linewidth=0.04](16.62,-1.406)(18.5,-1.406)
\psellipse[linecolor=black, linewidth=0.04, fillstyle=solid,fillcolor=black, dimen=outer](16.47,-1.4360001)(0.15,0.15)
\psellipse[linecolor=black, linewidth=0.04, fillstyle=solid,fillcolor=black, dimen=outer](17.19,-1.4360001)(0.15,0.15)
\psellipse[linecolor=black, linewidth=0.04, fillstyle=solid,fillcolor=black, dimen=outer](17.85,-1.4360001)(0.15,0.15)
\psellipse[linecolor=black, linewidth=0.04, fillstyle=solid,fillcolor=black, dimen=outer](18.51,-1.4360001)(0.15,0.15)
\psline[linecolor=black, linewidth=0.04, linestyle=dotted, dotsep=0.10583334cm](18.72,-1.406)(19.34,-1.406)
\psellipse[linecolor=black, linewidth=0.04, fillstyle=solid,fillcolor=black, dimen=outer](19.29,-1.4360001)(0.15,0.15)
\psline[linecolor=black, linewidth=0.04](19.44,-1.406)(20.16,-1.406)
\psellipse[linecolor=black, linewidth=0.04, fillstyle=solid,fillcolor=black, dimen=outer](20.05,-1.4360001)(0.15,0.15)
\psbezier[linecolor=black, linewidth=0.04](15.06,-1.406)(15.48,-1.106)(15.84,-0.866)(17.46,-0.8660000000000002)(19.08,-0.866)(19.62,-1.046)(19.98,-1.466)
\rput[bl](9.86,2.246){$C^0$}
\rput[bl](3.68,-2.854){$C^1$}
\rput[bl](9.86,-2.854){$C^2$}
\rput[bl](17.5,-2.854){$C^d$}
\psline[linecolor=black, linewidth=0.04, fillstyle=solid, doubleline=true, doublesep=0.02, doublecolor=black](1.58,-1.402)(1.628,0.758)(1.628,0.758)
\psline[linecolor=black, linewidth=0.04, doubleline=true, doublesep=0.02, doublecolor=black](4.412,-1.45)(4.92,0.694)
\psline[linecolor=black, linewidth=0.04, doubleline=true, doublesep=0.02, doublecolor=black](6.62,-1.402)(11.228,0.758)
\psline[linecolor=black, linewidth=0.04, doubleline=true, doublesep=0.02, doublecolor=black](3.692,-1.402)(14.492,0.71)
\psellipse[linecolor=black, linewidth=0.04, fillstyle=solid,fillcolor=black, dimen=outer](2.5886314,0.7740905)(0.17692351,0.1723982)
\psline[linecolor=black, linewidth=0.04](7.388,-1.354)(2.588,0.758)
\psline[linecolor=black, linewidth=0.04](15.02,-1.402)(3.98,0.614)
\psline[linecolor=black, linewidth=0.04](12.476,-1.45)(12.332,0.71)
\psline[linecolor=black, linewidth=0.04](20.06,-1.354)(13.676,0.758)
\psline[linecolor=black, linewidth=0.04](17.9,-1.45)(19.82,0.662)
\psline[linecolor=black, linewidth=0.04](17.132,-1.402)(10.412,0.71)
\rput[bl](4.244,-2.074){$v^1_{t}$}
\rput[bl](17.564,-2.25){$\small v^d_{\frac{n}{d}-t}$}
\rput[bl](1.364,-2.074){$\small v^1_{0}$}
\rput[bl](6.344,-2.25){$\small v^1_{\frac{n}{d}-1}$}
\rput[bl](14.744,-2.074){$\small v^d_{0}$}
\rput[bl](19.964,-2.25){$\small v^d_{\frac{n}{d}-1}$}
\end{pspicture}
}
\caption{generalized Petersen graph $GP(n,k)$, where $d=GCD(n,k)$.}
\end{figure}
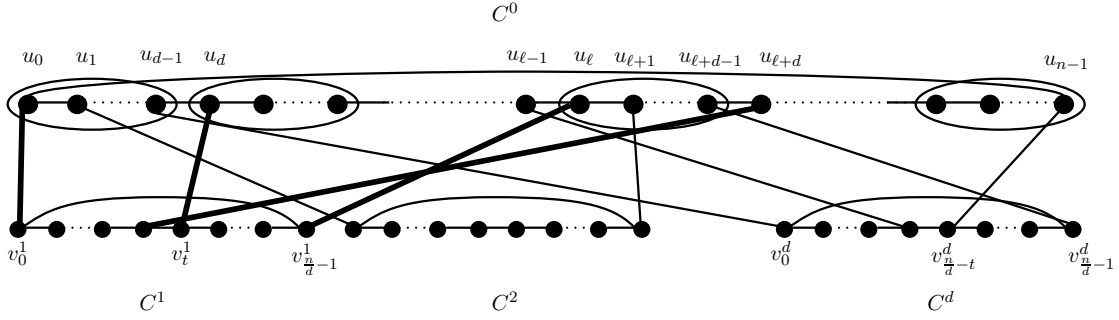

\begin{lemma}\label{lem:n/d5}
Let $n\geq 2k$  and $d=GCD(n,k)$. If $\frac{n}{d}\in\{2,5\}$, then $\chi^\prime_s(GP(n,k))\leq 5$.
\end{lemma}
\begin{proof}
{
First, suppose that $\frac{n}{d}=5$ ($n=5d$). If $k$ is a multiple of 5, then $d$ is also a multiple of 5,  and 
 we color subpaths of $C^0$ with the following patterns. Otherwise, we color $C^0$ with colors 0,1,2,3,4, repeatedly. 
\begin{align*}
&f(u_{n-1},u_{0},\ldots,u_{d-1})=0,1,2,3,4,\ldots,0,1,2,3,4,\\
  &f(u_{d-1},u_{d},\ldots,u_{2d-1})=1,2,3,4,0\ldots,1,2,3,4,0,\\
  &f(u_{2d-1},u_{2d},\ldots,u_{3d-1})=2,3,4,0,1\ldots,2,3,4,0,1,\\
    &f(u_{3d-1},u_{3d},\ldots,u_{4d-1})=3,4,0,1,2\ldots,3,4,0,1,2,\\
     &   f(u_{4d-1},u_{4d},\ldots,u_{n-1})=4,0,1,2,3\ldots,4,0,1,2,3.
\end{align*}
For every $0\leq i\leq n-1$, let $F_2(u_i)=\{f(u_{i-2}u_{i-1}),f(u_{i-1}u_{i}),f(u_{i}u_{i+1}), f(u_{i+1}u_{i+2})\}$.
It is easy to see that for every $0\leq i\leq n-1$, $F_2(u_i)$, $F_2(u_{i+k})$,..., $F_2(u_{i+4k})$ are pairwise different.
If we  choose a color of $\{0,1,2,3,4\}\setminus F_2(u_i)$ for the spoke incident to $u_i$, then five spokes with  different colors are incident to each  $C^j$, $1\leq j\leq d$. For every edge of $C^j$, we choose the color which does not appear in the  spokes in the distance at most two from it. Thus, we have a 5-star edge coloring  of $GP(n,k)$, as desired.

Now, suppose that $\frac{n}{d}=2$.
In this case, $n=2d$, $k=d$, and $C^1,\ldots,C^d$ are 1-paths. We first present an arbitrary 3-star edge coloring of  $C^0$, and  color spokes as follows. For every $1\leq i\leq d$,  we  set 
\[f(s^i_{0})=3+(i\mod{2}),~~
f(s^i_{1})=3+(i+1\mod{2}).\]
Obviously, two spokes in the distance two, have different colors, except for $s^1_{0},s^d_{1}$ and $s^1_{1},s^d_{0}$, when $ d$ is even. Thus,  it suffices to color $C^1$ and $C^d$ such that
 \[\{f(v^1_0v^1_1),f(v^d_0v^d_1)\}\cap \{f (u_0u_{n-1}),f(u_{d-1}u_{d})\}=\emptyset,\]
 which is possible, obviously.
 Now, to complete the star edge coloring,  it suffices to  choose an arbitrary color from $\{0,1,2,3,4\}\setminus\{f(s^i_0),f(s^i_1)\}$  for $C^i$, $2\leq i\leq d-1$. 
}
\end{proof}

\begin{theorem}\label{th:main}
If $GP(n,k)$ with $n\geq 2k$ is a generalized Petersen graph which $GCD(n,k)\geq 3$, then $\chi^\prime_s(GP(n,k))\leq 5$.
\end{theorem}

\begin{proof}
{By Lemma~\ref{lem:n/d5}, it suffices to prove the theorem for $\frac{n}{d}\not\in \{2,5\}$.
In this case, we are going to provide a 5-star edge coloring of $GP(n,k)$ such  cycles $C^0,C^1,\ldots,C^d$ use colors of $\{0,1,2\}$ and spokes use colors of $\{3,4\}$, except  in a few cases which we have to use color 3 or 4 for  at most one edge of $C^0$. Thus, we consider  subgraphs of $GP(n,k)$ and color them step by step. Actually, we consider cycles $C^1,\ldots,C^d$ , spokes, connector edges, and subpaths $Q_1:=u_{n-1},u_0,\ldots,u_d$, $Q_2:=u_{\ell-1},u_{\ell},\ldots,u_{\ell+d}$, $Q_3=u_d,u_{d+1},\ldots,u_{\ell-1}$, and $Q_4=u_{\ell+d},u_{\ell+d+1},\ldots,u_{n-1}$  of $C_0$. Then, we provide a star edge coloring of each subgraph such that the total coloring of $GP(n,k)$ is a 5-star edge coloring.
We propose a 5-star edge coloring $f:E(G)\rightarrow\{0,1,\ldots,4\}$ of these subgraphs in 6 steps, as follows.

\noindent{\bf Step 1.} Coloring the spokes.\\
We set
$f(s^i_{r})=3+(i+r\mod{2})$, for every $1\leq i\leq d$ and $0\leq r\leq \frac{n}{d}-1$.
The following facts for two spokes in the distance two hold.
\begin{itemize}
\item[(1)] For each $1\leq i\leq d$, spokes $s^i_{r}$ and $s^i_{r+1}$ receive the same color if and only if $\frac{n}{d}$ is odd and $r=\frac{n}{d}-1$.\label{f1}
\item[(2)] If $ 1\leq i<d$, then spokes $s^i_{r}$ and $s^{i+1}_{r}$ receive different colors.\label{f2}
\item[(3)] For every $r$, $0\leq r\leq \frac{n}{d}-1$, the distance of   $s^1_{r}$ and $s^d_{r-t}$ is two and  may be receive the same color. 
In other words, the adjacent spokes to each connector may be receive the same color. \label{f3}
\end{itemize}

In Table~\ref{SC}, we present the colors of some spokes which are adjacent to the connectors used in Step~4. More precisely, spokes $s_0^d$ and $s_t^1$ are  adjacent to  $u_{d-1}u_d$, spokes $s_{\frac{n}{d}-1-t}^d$ and $s_{\frac{n}{d}-1}^1$ are adjacent to  $u_{\ell-1}u_\ell$, spokes $s_{\frac{n}{d}-1}^d$ and $s_{t-1}^1$ are adjacent to $u_{\ell+d-1}u_{\ell+d}$, and spokes $s_{\frac{n}{d}-t}^d$ and $s_{0}^1$ are adjacent to $u_{n-1}u_0$ (see Figure~\ref{fig1}). Note that $GCD(t,\frac{n}{d})=1$, and if $\frac{n}{d}$ is even, tnen $t$ is odd.

\begin{table}[H]
\centering
\caption{Colors of some spokes.}
\label{SC}
\begin{tabular}{|c|c|c|c|c|c|c|c|c|c|c|}
\hline
 {\small $\frac{n}{d}$} & \small{$d$} & \small{$t$} & \small{$f(s^1_{0})$} & \small{$f(s^d_{0})$} & \small{$f(s^1_{t})$} & \small{$f(s^d_{\frac{n}{d}-1-t})$} & \small{$f(s^1_{\frac{n}{d}-1})$} & \small{$f(s^d_{\frac{n}{d}-1})$} & \small{$f(s^1_{t-1})$} & \small{$f(s^d_{\frac{n}{d}-t})$} \\ \hline 
\multirow{4}{*}{\small odd} & \multirow{2}{*}{\small odd} & {\small odd} &4 & 4 & 3 & 3 & 4 & 4 & 4 & 4  \\ \cline{3-11}

  &   & {\small even} &  4 &4 & 4 & 4 & 4 & 4 & 3 & 3 
   \\ \cline{2-11}
   & \multirow{2}{*}{\small even}  & {\small odd} & 4 & 3 & 3 & 4 & 4 & 3 & 4 & 3    \\ \cline{3-11}
     &   & {\small even} & 4 & 3 & 4 & 3 & 4 & 3 & 3 & 4 
     \\ \hline
     
     \multirow{2}{*}{\small even} & \multirow{1}{*}{\small odd} & {\small odd} & 4 & 4 & 3 & 4 & 3 & 3 & 4 & 3 \\ \cline{2-11}
       & \multirow{1}{*}{\small even}  & {\small odd} &4 & 3 & 3 & 3 & 3 & 4 & 4 & 4    \\ \hline
\end{tabular}
\end{table}

\noindent{\bf Step 2.} Coloring of $C^1$  and $C^d$.\\
Consider  two subpaths $P:v^1_{t-2}v^1_{t-1}v^1_{t}v^1_{t+1}$ and $P^\prime:v^1_{\frac{n}{d}-2}v^1_{\frac{n}{d}-1}v^1_{0}v^1_{1}$ of $C^1$. Set $f(P)=0,1,2$ and 
\begin{align*}
f(P^\prime)=\begin{cases}
1,0,1~~~t=1,\\
1,2,1~~~t=\frac{n}{d}-1,\\
2,1,0~~~t=2~\text{or}~t=\frac{n}{d}-2,\\
0,2,1~~~\text{otherwise}.
\end{cases}
\end{align*}
By  Lemma~\ref{lem1}, it is easy to see that these partial edge colorings of $C^1$ could be extended to a 3-star edge coloring of $C^1$. We color  $C^d$ such that $f(v^d_{i}v^d_{i+1})=f(v^1_{i+t}v^1_{i+1+t})$. Thus, $C^1$ and $C^d$ admits a 3-star edge coloring using $\{0,1,2\}$. 

\noindent{\bf Step 3.} Coloring the connector edges  adjacent to the spokes with  the same color.\\
Let $e=u_{rk-1}u_{rk}$,  $0\leq r\leq \frac{n}{d}-1$, be a connector edge which its adjacent spokes have the same color (i.e.  $f(s^1_r)=f(s^d_{r-t})$). To avoid any bicolored 4-path or 4-cycles, we should have 
\begin{equation}\label{eq1}
f(e)\not\in\mathcal{F}_{C^1}(v^1_r)\cup\mathcal{F}_{C^d}(v_{r-t}^d).
\end{equation}
We choose the color of $e$, from $\{0,1,2\}\setminus \mathcal{F}_{C^1}(v^1_r)$. Since we provide the coloring of $C^1$ in a clockwise $t$-shift pattern of the edge coloring of $C^d$,   $\mathcal{F}_{C^1}(v^1_{r})=\mathcal{F}_{C^d}(v^d_{r-t})\in\{0,1,2\}$, and condition (\ref{eq1}) holds.

Moreover, if $\frac{n}{d}$ is odd, then   $f(s^{1}_0)=f(s^1_{\frac{n}{d}-1})$ and $f(s^{d}_0)=f(s^d_{\frac{n}{d}-1})$. Hence, we should have 
\begin{align}
&f(v^1_0v^1_{\frac{n}{d}-1})\not\in\{f(u_{n-1}u_0),f(u_{\ell-1}u_{\ell})\},\label{e2}\\
&f(v^d_0v^d_{\frac{n}{d}-1})=f(v^1_{t-1}v^1_{t})\not\in\{f(u_{d-1}u_{d}),f(u_{\ell+d-1}u_{\ell+d})\}.\label{e3}
\end{align} 
 Because of coloring patterns of $P$ and $P^\prime$,   (\ref{e2}) and (\ref{e3}) are also hold. 


\noindent{\bf Step 4.} Coloring of paths $Q_1:=u_{n-1},u_0,\ldots,u_d$, and $Q_2:=u_{\ell-1},u_{\ell},\ldots,u_{\ell+d}$.\\
In this step, based on the parity of $\frac{n}{d}$, we consider two cases.

\noindent $\bullet$ Case 1: $\frac{n}{d}$ is odd.\\
If $\frac{n}{d}$ is odd, then $f(s^i_{0})=f(s^i_{\frac{n}{d}-1})$,  for every  $1\leq i\leq d$. To avoid any bicolored 4-path (or 4-cycle),  we should have 
\begin{equation}\label{l:first}
f(v^i_0v^i_{\frac{n}{d}-1})\not \in \mathcal{F}_{C^0}(u_{i-1})\cup\mathcal{F}_{C^0}(u_{\ell+i-1}).
\end{equation} 
 Also, we try not to use colors 3 and 4 as much as possible for the star edge coloring of $Q_1$  and $Q_2$ (at most once if necessary).
  For this purpose, it is sufficient to color $Q_1$  and $Q_2$ as similar as possible (except for at most one difference which is due to the use of color 3 or 4).
  
  We first determine the  coloring of the edges of $C^0$ which are adjacent to spokes  $s^1_{0}$, $s^1_{\frac{n}{d}-1}$,	 $s^d_{0}$, $s^d_{\frac{n}{d}-1}$.
 The edges of  paths  $P_1=u_{n-1},u_0,u_1$, $P_2=u_{d-2}u_{d-1}u_d$, $P_3=u_{\ell-1},u_{\ell},u_{\ell+1}$, and $P_4=u_{\ell+d-2},u_{\ell+d-1},u_{\ell+d}$ are adjacent to these spokes.
 Obviously, we should consider the colorings of $C^1$, $C^d$ and some spokes in  Table~\ref{SC}, to avoid creating any bicolored 4-path (4-cycle) after coloring the edges of $P_1$,$P_2$,$P_3$, and $P_4$. In addition, it must be noted that the obtained partial coloring could be extended  to a star edge coloring of  $Q_1:=u_{n-1},u_0,\ldots, u_d$ and $Q_2:=u_{\ell-1},u_\ell,\ldots,u_{\ell+d}$, using colors $\{0,1,2\}$. 
  In Table~\ref{EC}, we determine the coloring of $P_1$, $P_2$, $P_3$, and $P_4$ such that 
\begin{align}\label{const-2}
f(v^1_0v^1_{\frac{n}{d}-1})\not \in \mathcal{F}_{C^0}(u_{0})\cup\mathcal{F}_{C^0}(u_{\ell}), ~~f(v^d_0v^d_{\frac{n}{d}-1})\not \in \mathcal{F}_{C^0}(u_{d-1})\cup\mathcal{F}_{C^0}(u_{\ell+d-1}),
\end{align}
and 
\begin{align*}
&f(P_1)=f(P_3)~\text{or different in at most one connector color which is 3 or 4},\\
&f(P_2)=f(P_4)~\text{or different in at most one connector color which is 3 or 4}.
\end{align*}
 Note that exactly one connector is colored from each set of $\{u_0u_{n-1},u_{\ell-1}u_\ell\}$ and $\{u_{d-1}u_d,u_{\ell+d-1}u_{\ell+d}\}$ in Step~3 (see the colors of their adjacent spokes form Table~\ref{SC}). Therefore,  the above conditions could be fulfilled, easily.
 
In Table~\ref{EC}, for $t\in\{1,\frac{n}{d}\}$  and $d=5$ or  $d$ is even, we are forced to use color  3 or 4.
Actually, if we do not use any color of $\{3,4\}$ and just consider  constraints~(\ref{eq1}),(\ref{e2}),(\ref{e3}) for coloring $P_1$, $P_3$, $P_3$ and $P_4$, then we have

\begin{flalign*}
&f(P_1)=f(P_3)=\begin{cases}
2,1~~~~t=1, (d=5~\text{or}~ d~\text{is even})\\
0,1~~~~t=\frac{n}{d}-1, (d=5~\text{or}~d~\text{is even}).
\end{cases}
\\
&f(P_2)=\begin{cases}
0,2~~~~t=1,d=5,\\
2,0~~~~(t=\frac{n}{d}-1, d=5)~\text{or}~( t\in\{1,\frac{n}{d}-1\}, d~\text{is even}).
\end{cases}
\\
\vspace*{1cm}
&f(P_4)=\begin{cases}
f(P_2)~~~~t\in\{1,\frac{n}{d}-1\},d=5\\
0,2~~~~t\in\{1,\frac{n}{d}-1\},d~\text{is even}.
\end{cases}
\end{flalign*}

For $d=5$, by Lemma~\ref{lem:7-path}, it is impossible to extend  this partial coloring to a 3-star edge coloring of $u_{n-1},u_0,\ldots, u_5$ and $u_{\ell-1},u_\ell,\ldots,u_{\ell+5}$ with colors $\{0,1,2\}$. Thus, we are forced to use color 3 in the partial edge colorings of $Q_1$ and $Q_2$.
For even $d$, since $f(P_2)\neq f(P_4)$, the extensions of partial edge colorings of $Q_1$ and  $Q_2$ are different. Hence, for some $i$, $2\leq i\leq d-1$, $\mathcal{F}_{C^0}(u_{i-1})\neq \mathcal{F}_{C^0}(u_{\ell+i-1})$, 
$\{0,1,2\}\setminus( \mathcal{F}_{C^0}(u_{i-1})\cup \mathcal{F}_{C^0}(u_{\ell+i-1}))=\emptyset$,  and there is no color for  $v^i_0v^i_{\frac{n}{d}-1}$ to satisfies condition~(\ref{l:first}). For this reason, we use color 4 in the partial edge colorings of $Q_1$ and $Q_2$.

\noindent $\bullet$ Case 2: $\frac{n}{d}$ is even.\\
For even $\frac{n}{d}$,  $f(s^i_{0})\neq f(s^i_{\frac{n}{d}-1})$,  for $1\leq i\leq d$. Thus, our constraints for the partial edge coloring of $Q_1$ and $Q_2$ is less. We determine the color of edges with symbol (?) in Table~\ref{EC} in such a way, after applying  Process 1 or 2 (coloring functions $f_1$ or $f_2$) to extend  the partial edge colorings of $Q_1$ and $Q_2$, we obtain a 3-star edge coloring with the following properties. 
\begin{enumerate}
\item If $|\mathcal{F}_{C^0}(u_{d-2}u_{d-1})|\neq 1$, then  $|\mathcal{F}_{C^0}(u_{\ell+d-2}u_{\ell+d-1})|\neq 1$.
\item If  $|\mathcal{F}_{C^0}(u_{d-2}u_{d-1})|=1$, then  $|\mathcal{F}_{C^0}(u_{\ell}u_{\ell+1})|\neq 1$.
\end{enumerate} 
 These properties is useful in Step~5 to extend the star edge coloring of uncolored subpaths $Q_3=u_d,u_{d+1},...,u_{\ell-1}$ and $Q_4=u_{\ell+d},u_{\ell+d+1},...,u_{n-1}$ of $C^0$    with the the same direction.

 \begin{table}[H]
\footnotesize
\centering
\caption{ \small Colorings of   $P_1=u_{n-1},u_0,u_1$, $P_2=u_{d-2}u_{d-1}u_d$}{~~~~~~~~~\small  $P_3=u_{\ell-1},u_{\ell},u_{\ell+1}$, $P_4=\nolinebreak u_{\ell+d-2},u_{\ell+d-1},u_{\ell+d}$.}
\label{EC}
\begin{tabular}{|c|c|c|c|c|c|c|c|c|c|}

\hline
{\footnotesize $\frac{n}{d}$} & {\footnotesize $d$} & \multicolumn{2}{|c|}{\footnotesize{$t$}} &{ \footnotesize $f(P_1)$}  &{\footnotesize $f(P_2)$} & {\footnotesize $f(P_3)$} &{\footnotesize $f(P_4)$} \\ \hline 

 \multirow{12}{*}{{\footnotesize odd}} & \multirow{6}{*}{\footnotesize odd} & \multirow{3}{*}{\footnotesize odd}& {\footnotesize $t= 1$, $d=5$}&
3,1 & 0,2 & 2,1 &  0,3\\ \cline{4-8}
 &  &  & {\footnotesize $t=1$, $d\neq 5$} & 
  2,1 &0,2 & 2,1 & 0,2    \\ \cline{4-8}

  &  &  & {\footnotesize $t=\frac{n}{d}-2$} & 
  2,0 & 0,2 & 2,0 &0,2    \\ \cline{4-8}
   &  &  & {\footnotesize $t\neq 1,\frac{n}{d}-2$} & 
  0,1 & 0,2 & 0,1 & 0,2    \\ \cline{3-8}
  
   &   & \multirow{3}{*}{\footnotesize even} & {\footnotesize{$t=2$, $t\neq\frac{n}{d}-1$}} & 
        0,2 & 2,0 & 0,2&2,0    \\ \cline{4-8}
        &   &  & {\footnotesize{$t= \frac{n}{d}-1$}, $d= 5$} & 
        0,1 & 2,3 & 3,1 & 2,0  \\ \cline{4-8}
        &   &  & {\footnotesize{$t= \frac{n}{d}-1$}, $d\neq 5$} & 
        0,1 & 2,0 & 0,1& 2,0    \\ \cline{4-8}
        &   & &{\footnotesize $t\neq2, \frac{n}{d}-1$} & 1,0 & 2,0 & 1,0 & 2,0 
  \\ \cline{2-8}

   & \multirow{6}{*}{\footnotesize even}  & \multirow{3}{*}{\footnotesize odd}& {\footnotesize $t= 1$} &
  2,1 & 0,4 & 2,1 & 0,2    \\ \cline{4-8}
    &  &  & {\footnotesize $t=\frac{n}{d}-2$} & 
  0,2 &  2,0 & 0,2 &2,0    \\ \cline{4-8}
    &  & & {\footnotesize$t\neq1,\frac{n}{d}-2$} &   1,0 & 2,0 & 1,0 & 2,0  \\ \cline{3-8}
     &   & \multirow{3}{*}{\footnotesize even}&  {\footnotesize $t\neq \frac{n}{d}-1,2$} & 0,1 & 0,2 & 0,1 & 0,2
     \\ \cline{4-8}
     &   & & {\footnotesize{$t= 2$}, $t\neq\frac{n}{d}-1$ } & 
       2,0 & 0,2&2,0& 0,2    \\ \cline{4-8}
        &   & & {\footnotesize{$t= \frac{n}{d}-1$}} & 
       0,1 & 2,0 & 0,1& 2,4    \\ \hline
        
     \multirow{5}{*}{\footnotesize even} & \multirow{1}{*}{\footnotesize odd} & {\footnotesize odd}& & 2,1  & ?,0 & 0,1& ?,2     \\ \cline{2-8}


       & \multirow{3}{*}{\footnotesize even}  & \multirow{3}{*}{\footnotesize odd}& {\footnotesize $t\neq \frac{n}{d}-1,1$} & 0,1 & ?,0 & 1,2 & ?,2   \\ \cline{4-8}
            &  & & {\footnotesize $t=1$} & 2,0 & ?,0 & 2,1 & ?,2  \\ \cline{4-8}
                 & & & {\footnotesize $t=\frac{n}{d}-1$} & 
               0,1  & ?,0 & 0,2 & ?,2   \\ 
                \hline
               
%

\end{tabular}
\end{table}

  
%
 We now extend the partial star edge coloring of $Q_1$ and $Q_2$.
 If $t=1$, then $u_{0}=u_{\ell+d}$ and $u_{n-1}u_0\in Q_1\cap Q_2$. Thus, the direction of coloring is clockwise starting  from the left uncolored edges in $Q_2$ by Process 2 and then $Q_1$ by Process 1.  For odd $\frac{n}{d}$,  we just  extend the coloring of  $Q_2$ by Process 2, and then apply the same coloring pattern  for uncolored edges of  $Q_1$.
 
 If $t\neq 1$, then $Q_1\cap Q_2=\emptyset$  or $u_{d-1}u_d\in Q_1\cap Q_2$,  for $t=\frac{n}{d}-1$,
  .
In these cases,  the direction of coloring is clockwise  and starting from the left uncolored edges in $Q_1$ by Process 2 and then $Q_2$ by Process 1.  For odd $\frac{n}{d}$,  we first extend the coloring of  $Q_1$ by Process 2 and then apply the same coloring pattern  for  uncolored edges of $Q_2$.

  Note that,  for even $\frac{n}{d}$  the partial edge coloring of $Q_1$ and $Q_2$ in Table~\ref{EC} is not  in such a way that their extension can be  the same; hence we have to color $Q_1$ and $Q_2$,  separately.

\noindent{\bf Step 5.}  Coloring of $Q_3=u_d,u_{d+1},...,u_{\ell-1}$ and $Q_4=u_{\ell+d},u_{\ell+d+1},...,u_{n-1}$.\\
 According to the edge coloring of $Q_1$ and $Q_2$, we have two possibilities: $|\mathcal{F}_{C^0}(u_{d-2}u_{d-1})|=1$ or not.
 In the first case, by considering Table~\ref{EC} and the star edge coloring of $Q_1$ and $Q_2$, we have $|\mathcal{F}_{C^0}(u_0u_1)|\neq 1$, $|\mathcal{F}_{C^0}(u_{\ell}u_{\ell+1})|\neq 1$. 
 Thus, to avoid creating bicolored 4-path, we have unique choices for  $f(u_{d}u_{d+1})$ and may be $f(u_{\ell+d}u_{\ell+d+1})$. 
 For this reason, we start coloring  $Q_3$ and $Q_4$   from their left sides  and clockwise by Process 2. 
If $|\mathcal{F}_{C^0}(u_{d-2}u_{d-1})|\neq 1$, then $\mathcal{F}_{C^0}(u_{\ell+d-2}u_{\ell+d-1})|\neq 1$, and  we star coloring anti-clockwise from the right sides.
Considering the above direction, let $e\in Q_3$ and $e^\prime\in Q_4$ be the last connectors.
We first use Process 1 for uncolored edges of subpaths $Q_3\setminus\{e,e+1,\ldots,e+d-1\}$ and  $Q_4\setminus\{e^\prime,e^\prime+1,\ldots,e^\prime+d-1\}$.
If $d\geq 5$, then we use Process 2 for uncolored edges of subpaths $e,e+1,\ldots,e+d-1$ and $e^\prime,e^\prime+1,\ldots,e^\prime+d-1$.
If $d\in\{3,4\}$, then the color of  connectors $e$ and $e^\prime$ are crucial for the  extensibility of the partial edge coloring of $Q_3$ and $Q_4$ (Lemma~\ref{lem:7-path}). Thus, we consider the following states for the last connector $e_c\in\{e,e^\prime\}$.

\begin{itemize}
\item if $e_c$ is uncolored, then its adjacent spokes have different color. Hence,  
\begin{itemize}
\item if $d=3$, then choose $f(e_c)$ from $\{f(e_c+3),f(e_c+4)\}\setminus\{f(e_c-1)\}$.
\item if $d=4$, then choose $f(e_c)$ from $\{0,1,2\}\setminus\{f(e_c-1),f(e_c+5)\}$.
\end{itemize}
\item if $e_c$ is colored already, then its adjacent spokes have the same color such as $a$. Hence,  
\begin{itemize}
\item if $d=3$ and $f(e_c)\not\in \{f(e_c+3),f(e_c+4)\}$, then change  $f(e_c)$ to the color of   $\{3,4\}\setminus\{a\}$. 
\item if $d=4$ and  $f(e_c)=f(e_c+5)$, then change  $f(e_c)$ to the color of   $\{3,4\}\setminus\{a\}$. 
\end{itemize}
\end{itemize}
After that, we use Process 2  to extend the partial edge coloring of  $e,e+1,\ldots,e+d-1$ and $e^\prime,e^\prime+1,\ldots,e^\prime+d-1$, clockwise.
Since $f(e_c+1)\neq f(e_c-1)$,  path $e_c-1,e_c,e_c+1$ is not bicolored. 

\noindent{\bf Step 6.} Coloring of $C^i$, $2\leq i\leq d-1$.\\
If $\frac{n}{d}$ is odd, then  for every $1\leq i\leq d$,  
$f(v^i_0u_{i-1})=f(v^i_{\frac{n}{d}-1}u_{\ell+i-1})$, 
and $\mathcal{F}_{C^0}(u_{i-1})=\mathcal{F}_{C^0}(u_{\ell+i-1})$ (see Tables~\ref{SC}~and~\ref{EC}). Thus,  we choose $f(v^i_{0}v^i_{\frac{n}{d}-1})$ form $\{0,1,2\}\setminus \mathcal{F}_{C^0}(u_{i-1})$ to avoid any bicolored 4-path including spokes $s^i_0=v^i_0u_{i-1}$ and $s^i_{\frac{n}{d}-1}=v^i_{\frac{n}{d}-1}u_{\ell+i-1}$.

\noindent {\bf Claim}: at most one edge of $C^0$ is colored by 3 or 4.\\
In Step~5, we present a star edge coloring of $Q_3=u_d,u_{d+1},...,u_{\ell-1}$ and $Q_4=u_{\ell+d},u_{\ell+d+1},...,u_{n-1}$.
according to the direction of coloring of $Q_3$ and $Q_4$, the last colored connectors $e\in Q_3$ and $e^\prime\in Q_4$ may be either $\{u_{\ell-d-1}u_{\ell-d}, u_{n-d-1}u_{n-d}\}$ or $\{u_{2d-1}u_{2d},u_{\ell+2d-1}u_{\ell+2d}\}$. We prove that at most one connector of $\{e,e^\prime\}$ is  colored by 3 or 4, when $d\in\{3,4\}$. Note that 
\[\begin{cases}
Q_4=\emptyset~~~&t=1,\\
Q_3=\emptyset~~~&t=\frac{n}{d}-1,\\
Q_4=u_{n-d},\ldots,u_{n-1}~~~&2t=1,\\
Q_3=u_d,\ldots,u_{2d-1}~~~&2t=\frac{n}{d}-1.
\end{cases}\]
So, it is clear that in the above cases, $Q_3$  or $Q_4$ have no connector.
Thus, let $2t\not \in\{1,2,\frac{n}{d}-2,\frac{n}{d}-1\}$.

In Step~5,  we change the color of  the last connector  $e_c\in\{e^\prime,e\}$ to 3 or 4, if  its adjacent spokes have the same color  and one of the following conditions holds.
\begin{itemize}
\item[(i)] $d=3$ and $f(e_c)\not\in \{f(e_c+3),f(e_c+4)\}$.
\item[(ii)] $d=4$ and  $f(e_c)=f(e_c+5)$.
\end{itemize}
 Hence, assume that the color of adjacent spokes to $e_c$  are the same. We show that   for $d=3$, $\{f(e+3),f(e+4)\}=\{f(e^\prime+3),f(e^\prime+4)\}$ and for $d=4$,  $f(e+5)=f(e^\prime+5)$.
 For this purpose, first suppose that $\frac{n}{d}$ is odd. If $d=3$, then $P_1=e^\prime+3,e^\prime+4$, $P_3=e+3,e+4$ either or $P_2=e+3,e+4$, $P_4=e^\prime+3,e^\prime+4$. Hence, by Table~\ref{EC}, $\{f(e+3),f(e+4)\}=\{f(e^\prime+3),f(e^\prime+4)\}$.
 If $d=4$, then $e^\prime+5=u_0u_1$, $e+5=u_\ell u_{\ell+1}$ either or $e+5=u_{2}u_{3}$, $e^\prime+5=u_{\ell+2} u_{\ell+3}$. By Table~\ref{EC}, we also have  $f(e+5)=f(e^\prime+5)$.
 
 Now,  let $\frac{n}{d}$ be even.
If $d=3$, then $t$ should be odd ($t$ even is impossible). In this case,  $e=u_{2d-1}u_{2d}$ and $e^\prime=u_{\ell+2d-1}u_{\ell+2d}$.  Moreover, $e$ and $e^\prime$ are adjacent to the spokes of $\{s^0_{2t},s^d_{t}\}$ and $\{s^0_{2t-1},s^d_{t-1}\}$, respectively. Since the parities of  indices of each pair of spokes are different, they have different colors (see Step~1). Thus,  we do not use color 3 or 4 for $e$ and $e^\prime$ in Step~5.
 
 For even $\frac{n}{d}$ and  $d=4$, we have two probabilities; $t$ is odd or even. 
If $t$ is odd, then $f(Q_1)=0,1,2,1,0$ and $f(Q_2)=1,2,0,1,2$. Thus, $\{e,e^\prime\}=\{u_{7}u_{8},u_{\ell+7}u_{\ell+8}\}$, $f(e+5)=f(u_2,u_3)=1$, and $f(e^\prime+5)=f(u_{\ell+2},u_{\ell+3})=1$.
If $t$ is even, then $f(Q_1)=2,1,0,2,0$ and $f(Q_2)=0,1,2,0,2$. Then,  $\{e,e^\prime\}=\{u_{\ell-5}u_{\ell-4},u_{n-5}u_{n-4}\}$, $f(e+5)=f(u_{0},u_{1})=1$, and $f(e^\prime+5)=f(u_{\ell},u_{\ell+1})=1$.

Consequently,  $\{f(e+3),f(e+4)\}=\{f(e^\prime+3),f(e^\prime+4)\}$ for $d=3$ and $f(e+5)=f(e^\prime+5)$, for $d=4$. Now, it suffices to show that $f(e)\neq f(e^\prime)$ to prove condition $\rm{(i)}$ either or $\rm{(ii)}$ holds for at most one of the last connectors. Note that
 \begin{itemize}
 \item[\rm{(a)}] if $e=u_{2d-1}u_{2d}$ and $e^\prime=u_{\ell+2d-1}u_{\ell+2d}$, then $f(e)\in\{0,1,2\}\setminus\mathcal{F}_{C^1}(v^1_{2t})$ and $f(e^\prime)\in\{0,1,2\}\setminus\mathcal{F}_{C^1}(v^1_{2t-1})$. 
 \item[\rm{(b)}]
 if  $e=u_{\ell-d-1}u_{\ell-d}$ and $e^\prime=u_{n-d-1}u_{n-d}$, then $f(e)\in\{0,1,2\}\setminus\mathcal{F}_{C^1}(v^1_{\frac{n}{d}-1-t})$ and $f(e^\prime)\in\{0,1,2\}\setminus\mathcal{F}_{C^1}(v^1_{\frac{n}{d}-t})$.
 \end{itemize} 
 If $f(e)=f(e^\prime)$, then we should have $\mathcal{F}_{C^1}(v^1_{2t})= \mathcal{F}_{C^1}(v^1_{2t-1})$ in case (a) and $\mathcal{F}_{C^1}(v^1_{\frac{n}{d}-1-t})=\mathcal{F}_{C^1}(v^1_{\frac{n}{d}-t})$ in case (b). 
 Considering the edge coloring of $C^1$ in Step~2 by Lemma~\ref{lem1}, if $\mathcal{F}_{C^1}(v^1_{2t})= \mathcal{F}_{C^1}(v^1_{2t-1})$ and $\mathcal{F}_{C^1}(v^1_{\frac{n}{d}-1-t})=\mathcal{F}_{C^1}(v^1_{\frac{n}{d}-t})$, then 
 $2t\in\{0,1,t-1,t,t+1,\frac{n}{d}-1\}$, which is not possible.
Therefore, colorings of  $C^1$ and $C^d$ cause  we  use colors of $\{3,4\}$ for   $e$ either or  $e^\prime$ in Step~5. This proves Claim.

Now, we are ready to provide a 3-star edge coloring for every $C^i$, $2\leq i\leq d-1$. If   the star edge coloring $C^0$ does not  use any color of $\{3,4\}$, then there is no additional constraint on the star edge coloring of $C^i$, $2\leq i\leq d-1$. 
 In this case, since there exists at most one colored edge in $C^i$, $2\leq i\leq d-1$,  it is easy to see that the partial edge coloring of   $C^i$ can be extended to a 3-star edge coloring with colors of $\{0,1,2\}$. Thus, assume that we have  a connector $u_iu_{i+1}$ with color 3 or 4 from 
\[\{u_{n-1}u_0, u_{d-1}u_d, u_{\ell+d-1}u_{\ell+d}\}~~\text{(see Table~\ref{EC})}\]
or
\[\{u_{\ell-d-1}u_{\ell-d}, u_{n-d-1}u_{n-d},u_{2d-1}u_{2d},u_{\ell+2d-1}u_{\ell+2d}\}~~ \text{(for~$d\in\{3,4\}$ in Step~5)}.\] 
Thus, spokes $u_{i-1}v^{d-1}_{t(\frac{i+1}{d}-1)}$ and $u_{i+2}v^{2}_{t\frac{i+1}{d}}$  are in distance two from $u_iu_{i+1}$ and have the same color as $f(u_iu_{i+1})$. To avoid any bicolored 4-path containing  bicolored 3-path $u_i,u_{i+1},u_{i+2},v^{2}_{t\frac{i+1}{d}}$ or $v^{d-1}_{t(\frac{i+1}{d}-1)},u_{i-1},u_i,u_{i+1}$, we should have
\begin{enumerate}
\item[$\bullet$] $f(u_{i-1}u_{i})\not \in\mathcal{F}_{C^{d-1}}(v^{d-1}_{t(\frac{i+1}{d}-1)})$,
\item[$\bullet$] $f(u_{i+1}u_{i+2})\not \in\mathcal{F}_{C^2}(v^{2}_{t\frac{i+1}{d}})$,
\end{enumerate}
To satisfy the above conditions, we properly color the adjacent edges to $v^{2}_{t\frac{i+1}{d}}$ and $v^{d-1}_{t(\frac{i+1}{d}-1)}$ using colors of  $\{0,1,2\}\setminus\{f(u_{i+1}u_{i+2})\}$ and $\{0,1,2\}\setminus \{f(u_{i-1}u_{i} )\}$, respectively. 
 Now, using Lemmas~\ref{lem2}~and~\ref{lem3}, we extend the partial star edge colorings of $C^2$ and $C^{d-1}$ to a 3-star edge coloring using colors $\{0,1,2\}$.
}
\end{proof}

\begin{theorem}
Let  $n\geq 2k$, $GCD(n,k)=2$, and $t$ be the minimum positive integer which $tk=2\pmod{n}$.  Then, $GP(n,k)$ admits a 5-star edge coloring in the following cases.
\begin{itemize}
\item[\rm{(}$1$\rm{)}]
$n = 0 \pmod 6$.
\item[\rm{(}$2$\rm{)}]
$n=2\pmod 6$, $t=2\pmod 3$.
\item[\rm{(}$3$\rm{)}]
$n=4\pmod 6$, $t=1\pmod 3$.
\end{itemize}
\end{theorem}

\begin{proof}
{ Since $GCD(n,k)=2$, we have three disjoint cycles $C^0$, $C^1$ and $C^2$, which the length of $C^0$ is $n$ and the length of  $C^1$ and $C^2$ is $\frac{n}{2}$. Moreover, the distance of every two consecutive connectors is two.
We color the spokes of $GP(n,k)$ applying Step~1 of Theorem~\ref{th:main}.
For every $1\leq i\leq \frac{n}{d}$ and $j\in\{1,2\}$,  the ends of every spokes $s^j_i$ and $s^j_{i+t}$ on $C^0$, are in distance three. Furthermore, two spokes $s^0_i$ and $s^1_{i-t}$ are adjacent to the different ends of  a common connector edge.

We  provide a 3-star edge coloring of  $C^1$  with colors of $\{0,1,2\}$ and  then we color $C^2$ in a anticlockwise $t$-shift pattern of  $C^1$.
Thus, in each case of the theorem we just determine the coloring pattern of $C^1$. After coloring of $C^1$, we should color $C^0$ such that  $G(n,p)$ has no bicolored 4-path.  For this purpose, we apply Step~3 of Theorem~\ref{th:main} to color the connectors  adjacent to spokes with the same color.  Note that, for even $t$, connectors $u_{n-1}u_0$ and $u_{\ell+1}u_{\ell+2}$, and for odd $t$, connectors $u_1u_2$ and $u_{\ell-2}u_{\ell-1}$, are adjacent to the spokes with the same color. Moreover, if $\frac{n}{2}$ is odd, then we should satisfy conditions~(\ref{const-2}) in Step~4 of Theorem~\ref{th:main} for coloring of the edges incident to $\{u_0,u_1,u_\ell,u_{\ell+1}\}$ to have no bicolored 4-path. Thus, we first color $C^1$ and then  $P_1=u_{n-1},u_0,u_1,u_2$ and $P_2=u_{\ell-1},u_\ell,u_{\ell+1},u_{\ell+2}$ as follows.

\noindent{\bf Case \rm{(}1\rm{)}}:  $n = 0 \pmod 6$.\\
 Since $GCD(t,\frac{n}{2})=1$,   we have $t\neq 0\pmod 3$. and color $C^1$ such that
\[~~f(C^1)=
\begin{cases}
2,\underbrace{\underline{ 1,2,0},\ldots,\underline{1,2,0}}_{\frac{n}{2}-3},1,2~~~ &t=1\pmod3,\\
0,\underbrace{\underline{ 1,2,0},\ldots,\underline{1,2,0}}_{\frac{n}{2}-3},2,1~~~ &t=2\pmod3.
\end{cases}
\]
Thus, for every $0\leq i\leq \frac{n}{d}-1$, where $i\neq 1, \frac{n}{2}-1-t$, we have  $\mathcal{F}(v^1_i)\neq \mathcal{F}(v^1_{i+t})$.
After applying Step~3 of Theorem~\ref{th:main}, for $i\in\{1,\frac{n}{2}-1-t\}$,  $f(u_{ki-1}u_{ki})\in \{0,1,2\}\setminus\mathcal{F}(v^1_i)$ either or $f(u_{ki-1}u_{ki}+2)\in \{0,1,2\}\setminus\mathcal{F}(v^1_{i+t})$.  

Now, for $t=1\pmod 3$, we set
\[f(P_1)=\begin{cases}
1,2,0~~~&t~\text{is even},\\
0,1,2~~~ &t~\text{is odd}.\\
\end{cases}~~~~
f(P_2)=\begin{cases}
0,2,1~~~&t~\text{is even}\\
2,1,0~~~ &t~\text{is odd}.
\end{cases}
\]

\noindent For $t=2\pmod 3$, we set
\[f(P_1)=\begin{cases}
2,0,1~~~&t~\text{is even},\\
1,2,0~~~ &t~\text{is odd}.
\end{cases}
~~~~f(P_2)=\begin{cases}
1,0,2~~~&t~\text{is even},\\
0,2,1~~~ &t~\text{is odd}.
\end{cases}
\]                  
\noindent{\bf Case \rm{(}$2$\rm{)}}: $n=2\pmod 6$, $t=2\pmod 3$.\\
In this case we set $f(C^1)=\underbrace{\underline{0, 1,2},\ldots,\underline{0,1,2}}_{\frac{n}{2}-1},1$.
Hence, for every $0\leq i\leq \frac{n}{d}-1$, $\mathcal{F}(v^1_i)\neq \mathcal{F}(v^1_{i+t})$. Now, we color $P_1$ and $P_2$ as follows.
\[f(P_1)=\begin{cases}
2,0,1~~~&t~\text{is even},\\
1,2,0~~~ &t~\text{is odd}.
\end{cases}
~~~~
f(P_2)=\begin{cases}
2,1,0~~~&t~\text{is even},\\
0,1,2~~~ &t~\text{is odd}.
\end{cases}
\]    

In Cases ($1$) and ($2$), because of coloring of $P_1$ and $P_2$, it is possible to have two consecutive connectors $e$ and $e+2$ with the same color. If $f(e)=f(e+2)$, then  $f(e-2)=f(e+4)$ and we could extend this partial edge coloring of   $e-2,e-1,...,e+4$ to a 3-star edge coloring of it. Note that, if $e-2$ either or $e+4$ are not colored, we determine its color such that the condition $f(e-2)=f(e+4)$ holds.

\noindent{\bf Case \rm{(}$3$\rm{)}}: $n=4\pmod 6$, $t=1\pmod 3$.
$$f(C_1)=\underbrace{\underline{ 0,1,2},\ldots,\underline{0,1,2}}_{\frac{n}{2}-5},0,1,0,2,1.$$
In this case, for every $0\leq i\leq \frac{n}{d}-1$ and $i\not\in\{0, \frac{n}{2}-t,\frac{n}{2}-3,\frac{n}{2}-3-t\}$, $\mathcal{F}(v^1_i)\neq \mathcal{F}(v^1_{i+t})$. 
 Let $e=u_0u_{n-1}$ and $e^\prime=u_{k(\frac{n}{2}-3)-1}u_{k(\frac{n}{2}-3)}$. Then,  at least  two consecutive  connectors of $\{e-2,e,e+2\}$ and $\{e^\prime-2,e^\prime,e^\prime+2\}$ are colored with 2 after applying Step~3 of Theorem~\ref{th:main}. Let the color of spokes adjacent to $e$ and $e^\prime$ are $c$ and $c^\prime$, respectively.
We change the color of spokes adjacent to $e$ and $e^\prime$ to 2 and set $f(e)=0$, $f^\prime(e^\prime)=1$. Then, we choose color 2 for $\{e-2,e+2,e^\prime-2, e^\prime+2\}$, color $c$ for $\{e-1,e+1\}$, and color $c^\prime$ for $\{e^\prime-1,e^\prime+1\}$.

Now, in each cases of the theorem, we color remaining uncolored connectors such that every two consecutive connectors have different colors. After that, it is easy to complete the star edge coloring of $C^0$ by choosing the color of each remaining edge differently from the colors of its adjacent connectors.  
}
\end{proof}

{\setlength{\baselineskip}{0.73\baselineskip}

}
\end{document}